\documentclass[12pt,lenq]{amsart}
\usepackage{amsmath}
\usepackage{amscd}
\usepackage{amssymb}
\usepackage{amsbsy}
\usepackage{amsfonts}
\usepackage{latexsym}
\usepackage{graphics}
\usepackage{graphicx}
\usepackage{amsmath,amscd,latexsym}
\usepackage{multirow}
\usepackage{array}
\usepackage{paralist}
\usepackage{titletoc}

\theoremstyle{plain}
\newtheorem{theorem}{Theorem}[section]

\newtheorem{lemma}[theorem]{Lemma}
\newtheorem{corollary}[theorem]{Corollary}

\newtheorem{definition}[theorem]{Definition}

\newtheorem{main theorem}[theorem]{Main Theorem}

\newtheorem{convention}[theorem]{Convention}

\newcommand{\ZZ}{\mathbb{Z}}
\newcommand{\QQ}{\mathbb{Q}}
\newcommand{\RR}{\mathbb{R}}

\newcommand{\HH}{\mathbb{H}}
\newcommand{\QQQ}{\hat{\mathbb{Q}}}
\newcommand{\RRR}{\hat{\mathbb{R}}}

\newcommand{\DD}{\mathcal{D}}
\newcommand{\RGPC}[2]{\Gamma({#1};{#2})}
\newcommand{\RGPP}[1]{\hat\Gamma_{#1}}
\newcommand{\RGP}[1]{\Gamma_{#1}}

\newcommand{\Hecke}{\mbox{$G$}}

\newcommand{\orbs}{\mbox{\boldmath$S$}}

\newcommand{\svert}{\,|\,}

\newcommand{\lp}{(\hskip -0.07cm (}
\newcommand{\rp}{)\hskip -0.07cm )}

\makeatletter
\renewcommand\subsection{\@startsection{subsection}{2}{0mm}
    {-10.5dd plus-8pt minus-4pt}{10.5dd}
     {\normalsize\upshape}}
\makeatother

\begin{document}

\title{Homotopically equivalent simple loops
on 2-bridge spheres in Heckoid orbifolds for 2-bridge links (II)}

\author{Donghi Lee}
\address{Department of Mathematics\\
Pusan National University \\
San-30 Jangjeon-Dong, Geumjung-Gu, Pusan, 609-735, Republic of Korea}
\email{donghi@pusan.ac.kr}

\author{Makoto Sakuma}
\address{Department of Mathematics\\
Graduate School of Science\\
Hiroshima University\\
Higashi-Hiroshima, 739-8526, Japan}
\email{sakuma@math.sci.hiroshima-u.ac.jp}

\subjclass[2010]{Primary 20F06, 57M25\\
\indent {The first author was supported by Basic Science Research Program
through the National Research Foundation of Korea(NRF) funded
by the Ministry of Education, Science and Technology(2012R1A1A3009996).
The second author was supported
by JSPS Grants-in-Aid 22340013.}}

\begin{abstract}
In this paper and its prequel,
we give a necessary and sufficient condition
for two essential simple loops on a $2$-bridge sphere
in an even Heckoid orbifold for a $2$-bridge link
to be homotopic in the orbifold.
We also give a necessary and sufficient condition
for an essential simple loop on a $2$-bridge sphere
in an even Heckoid orbifold for a $2$-bridge link
to be peripheral or torsion in the orbifold.
The prequel treated the case when
the $2$-bridge link is a $(2,p)$-torus link,
and this paper treats the remaining cases.
\end{abstract}
\maketitle


\section{Introduction}
Let $K(r)$ be the $2$-bridge link of slope $r \in \QQ$
and let $n$ be an integer or a half-integer greater than $1$.
Also let $\orbs(r;n)$ be the Heckoid orbifold of index $n$ for $K(r)$,
and let $\Hecke(r;n)$ be the Heckoid group of index $n$ for $K(r)$
which is the orbifold fundamental group of $\orbs(r;n)$.
According to whether $n$ is an integer or a non-integral half-integer,
the Heckoid group $\Hecke(r;n)$ and the Heckoid orbifold $\orbs(r;n)$
are said to be {\it even} or {\it odd}.

The purpose of the present paper and its prequel~\cite{lee_sakuma_9}
is to give the following complete solution to the natural question
proposed in \cite[Question~2.1]{lee_sakuma_9}
for even Heckoid orbifolds,
where we use the same notation and terminology as in
\cite{lee_sakuma_9} without specifically mentioning.

\begin{main theorem}
\label{thm:conjugacy}
Suppose that $r$ is a non-integral rational number and
that $n$ is an integer greater than $1$.
Then the following hold.
\begin{enumerate}[\indent \rm (1)]
\item The simple loops $\{\alpha_s \svert s\in I(r;n)\}$ represent
mutually distinct conjugacy classes in $\Hecke(r;n)$.

\item There is no rational number $s \in I(r;n)$
for which $\alpha_s$ is peripheral in $\Hecke(r;n)$.

\item There is no rational number $s \in I(r;n)$
for which $\alpha_s$ is torsion in $\Hecke(r;n)$.
\end{enumerate}
\end{main theorem}

In the prequel~\cite{lee_sakuma_9},
we treated the case when $r\equiv \pm1/p \pmod{1}$ for some integer $p \ge 2$.
And this paper treats the remaining cases.
The key tool used in the proofs is small cancellation theory
applied to the upper presentations of even Heckoid groups.

This paper is organized as follows.
In Section~\ref{sec:technical_lemmas},
we establish technical lemmas which will play essential roles in the succeeding sections.
Sections~\ref{sec:proof of main theorem(1) for the general case} and \ref{sec:proof of main theorem(3)}
are devoted to the proof of Main Theorem~\ref{thm:conjugacy}.

\section{Technical Lemmas}
\label{sec:technical_lemmas}

In the remainder of this paper unless specified otherwise,
suppose that $r$ is a rational number with $0<r<1$
such that $r \neq 1/p$ for any integer $p \ge 2$,
and let $n$ be an integer with $n \ge 2$.
Write $r$ as a continued fraction expansion $r=[m_1, m_2, \dots, m_k]$, where $k \ge 2$,
$(m_1, \dots, m_k) \in (\mathbb{Z}_+)^k$ and $m_k \ge 2$.
For brevity, we often write $m$ for $m_1$.
Recall that the region, $R$, bounded by a pair of
Farey edges with an endpoint $\infty$
and a pair of Farey edges with an endpoint $r$
forms a fundamental domain for the action of $\RGPC{r}{n}$ on $\HH^2$
(see \cite[Figure~1]{lee_sakuma_9}).
Let $I_1(r;n)$ and $I_2(r;n)$ be the (closed or half-closed) intervals in $\RR$
defined as follows:
\[
\begin{aligned}
I_1(r;n) &=
\begin{cases}
[0, r_1], \ \mbox{where} \ r_1=[m_1, \dots, m_{k-1}, m_k-1, 2] , & \mbox{if $k$ is even,}\\
[0, r_1), \ \mbox{where} \ r_1=[m_1, \dots, m_k, 2n-2], & \mbox{if $k$ is odd,}
\end{cases}\\
I_2(r;n) &=
\begin{cases}
(r_2,1], \ \mbox{where} \ r_2=[m_1, \dots, m_k, 2n-2], & \mbox{if $k$ is even,}\\
[r_2, 1], \ \mbox{where} \ r_2=[m_1, \dots, m_{k-1}, m_k-1, 2], & \mbox{if $k$ is odd.}
\end{cases}
\end{aligned}
\]
Then we may choose a fundamental domain $R$ so that
the intersection of $\bar R$ with $\partial \HH^2$ is equal to
the union $\bar I_1(r;n) \cup \bar I_2(r;n)\cup \{\infty,r\}$.

\subsection{The case when $s \in I_1(r;n) \cup I_2(r;n)$}

In this subsection, we investigate important properties of $CS(s)$
for a rational number $s$ such that $s \in I_1(r;n) \cup I_2(r;n)$.
These properties will be used in the proof of Main Theorem~\ref{thm:conjugacy}
in the succeeding sections.
The following lemma is a slight refinement of \cite[Proposition~5.1]{lee_sakuma_7}.

\begin{lemma}
\label{lem:connection}
Let $S(r)=(S_1, S_2, S_1, S_2)$ be as in \cite[{\it Lemma}~3.9]{lee_sakuma_9}.
Then for any rational number $s \in I_1(r;n) \cup I_2(r;n)$, the following hold.
\begin{enumerate}[\indent \rm (1)]
\item If $k$ is even, then $CS(s)$ does not contain
$((2n-2) \langle S_1, S_2 \rangle, S_1)$ as a subsequence.

\item If $k$ is odd, then $CS(s)$ does not contain
$((2n-2) \langle S_2, S_1 \rangle, S_2)$ as a subsequence.
\end{enumerate}
\end{lemma}

\begin{proof}
We prove (1) and (2) simultaneously by induction on $k \ge 2$.
As we declared at the beginning of this section, we write $m$ for $m_1$, for simplicity.
By \cite[Lemma~3.9]{lee_sakuma_9}, $S_1$ begins and ends with $m+1$,
and $S_2$ begins and ends with $m$.
Suppose on the contrary that there exists some $s \in I_1(r;n) \cup I_2(r;n)$
for which $CS(s)$ contains $((2n-2) \langle S_1, S_2 \rangle, S_1)$ as a subsequence
provided $k$ is even
and $((2n-2) \langle S_2, S_1 \rangle, S_2)$ as a subsequence
provided $k$ is odd.
This implies by \cite[Lemma~3.5]{lee_sakuma_9} that
$CS(s)$ consists of $m$ and $m+1$.
So $s \neq 0$ and $s$ has a continued fraction expansion $s=[l_1, \dots, l_t]$, where
$t \ge 2$, $(l_1, \dots, l_t) \in (\mathbb{Z}_+)^t$, $l_1=m$ and $l_t \ge 2$.
For the rational numbers $r$ and $s$, define the rational numbers
$\tilde{r}$ and $\tilde{s}$ as in \cite[Lemma~3.8]{lee_sakuma_9}
so that $CS(\tilde{r})=CT(r)$ and $CS(\tilde{s})=CT(s)$.

We consider three cases separately.

\medskip
\noindent {\bf Case 1.} $m_2=1$.
\medskip

In this case, $k \ge 3$ and, by \cite[Corollary~3.14(1)]{lee_sakuma_9},
$(m+1, m+1)$ appears in $S_1$, so in $CS(s)$, as a subsequence.
Thus by \cite[Lemma~3.5]{lee_sakuma_9}, $l_2=1$ and so $t \ge 3$.
So, we have
\[
\tilde{r}=[m_3, \dots, m_k] \quad \text{\rm and} \quad
\tilde{s}=[l_3, \dots, l_t].
\]
It follows from $s \in I_1(r;n) \cup I_2(r;n)$ that
$\tilde{s} \in I_1(\tilde{r};n) \cup I_2(\tilde{r};n)$.
At this point, we divide this case into three subcases.

\medskip
\noindent {\bf Case 1.a.} $k=3$.
\medskip

By \cite[Lemma~3.12(1)]{lee_sakuma_9}, $S_1=(m_3\langle m+1 \rangle)$ and $S_2 =(m)$.
Since $((2n-2) \langle S_2, S_1 \rangle, S_2)$ is contained in $CS(s)$ by the assumption,
this implies that $CS(\tilde{s})=CT(s)$ contains
$((2n-2) \langle m_3 \rangle)$ as a subsequence.
But since $\tilde{r}=1/m_3=[m_3]$ and $\tilde{s} \in I_1(\tilde{r};n) \cup I_2(\tilde{r};n)$,
this gives a contradiction to \cite[Lemma~5.1]{lee_sakuma_9}.

\medskip
\noindent {\bf Case 1.b.} $k\ge 4$ and $k$ is even.
\medskip

Let $S(\tilde{r})=(T_1, T_2, T_1, T_2)$ be the decomposition of $S(\tilde{r})$
given by \cite[Lemma~3.9]{lee_sakuma_9}.
Since $S_1$ begins and ends with $m+1$, $S_2$ begins and ends with $m$,
and since $((2n-2) \langle S_1, S_2 \rangle, S_1)$ is contained in $CS(s)$ by the assumption,
we see by \cite[Lemma~3.12(2)]{lee_sakuma_9} that
$CS(\tilde{s})=CT(s)$ contains, as a subsequence,
\[
(t_1+\ell', t_2, \dots, t_{s_1-1}, t_{s_1}, T_2, (2n-3) \langle T_1, T_2 \rangle,
t_1, t_2, \dots, t_{s_1-1}, t_{s_1}+\ell''),
\]
where $(t_1, t_2, \dots, t_{s_1})=T_1$ and $\ell', \ell'' \in\ZZ_+\cup\{0\}$.
Since $t_1=t_{s_1}=m_3+1$ by \cite[Lemma~3.9]{lee_sakuma_9},
this actually implies that $\ell'=\ell''=0$, and therefore
$CS(\tilde{s})$ contains
$((2n-2) \langle T_1, T_2 \rangle, T_1)$ as a subsequence.
But since $\tilde{r}=[m_3, \dots, m_k]$ and
$\tilde{s} \in I_1(\tilde{r};n) \cup I_2(\tilde{r};n)$,
this gives a contradiction to the induction hypothesis.

\medskip
\noindent {\bf Case 1.c.} $k\ge 4$ and $k$ is odd.
\medskip

Let $S(\tilde{r})=(T_1, T_2, T_1, T_2)$ be the decomposition of $S(\tilde{r})$
given by \cite[Lemma~3.9]{lee_sakuma_9}.
Since $S_1$ begins and ends with $m+1$, $S_2$ begins and ends with $m$,
and since $((2n-2) \langle S_2, S_1 \rangle, S_2)$ is contained in $CS(s)$ by the assumption,
we see by \cite[Lemma~3.12(2)]{lee_sakuma_9} that
$CS(\tilde{s})=CT(s)$ contains $((2n-2) \langle T_2, T_1 \rangle, T_2)$ as a subsequence.
But since $\tilde{r}=[m_3, \dots, m_k]$ and
$\tilde{s} \in I_1(\tilde{r};n) \cup I_2(\tilde{r};n)$,
this gives a contradiction to the induction hypothesis.

\medskip
\noindent {\bf Case 2.} $k=2$ and $m_2=2$.
\medskip

In this case, $r=[m,2]$, so by \cite[Lemma~3.12(3)]{lee_sakuma_9}, $S_1=(m+1)$ and $S_2=(m)$.
Since $((2n-2) \langle S_1, S_2 \rangle, S_1)$ is contained in $CS(s)$ by the assumption,
$((2n-2) \langle m+1, m \rangle, m+1)$ is contained in $CS(s)$.
This implies that $CS(\tilde{s})=CT(s)$ contains
$((2n-2) \langle 1 \rangle)$ as a subsequence.
Moreover, we can see that this subsequence is proper,
i.e., it is not equal to the whole cyclic sequence $CS(\tilde{s})=CT(s)$.
As described below, this in turn implies that
$s$ has the form
either
$s=[m,1, 1, l_4 \dots, l_t]$
or $s=[m,2, l_3, \dots, l_t]$ with $l_3 \ge 2n-2$.
If $l_2=1$, then $\tilde s=[l_3,\dots,l_t]$ and so
$l_3$ is the minimal component of $CS(\tilde s)$
(see \cite[Lemma~3.5]{lee_sakuma_9}).
Hence we must have $l_3=1$,
i.e.,
$s=[m,1, 1, l_4 \dots, l_t]$,
because $CS(\tilde{s})$ contains $1$ as a component.
On the other hand, if $l_2\ge 2$, then $\tilde s=[l_2-1,\dots,l_t]$ and so
$l_2-1$ is the minimal component of $CS(\tilde s)$.
Since $CS(\tilde{s})$ contains $1$ as a component,
we have
$l_2-1=1$, i.e., $l_2=2$.
Since $CS(\tilde{s})$ contains $((2n-2) \langle 1 \rangle)$ as a
subsequence,
we see that
$CS(\tilde{\tilde s})=CT(\tilde s)$ contains a component $\ge 2n-2$.
Since the subsequence $((2n-2) \langle 1 \rangle)$ of
$CS(\tilde{s})$ is proper, we see $t\ge 3$ and $l_3\ge 2$.
Thus $\tilde{\tilde s}=[l_3-1,\dots,l_t]$ and therefore
$l_3-1$ is the minimal component of $CS(\tilde{\tilde s})$.
Hence we must have
$l_3=(l_3-1)+1\ge 2n-2$
and so $s=[m,2, l_3, \dots, l_t]$ with $l_3 \ge 2n-2$.

But then $s$ cannot belong to the interval
$I_1(r;n) \cup I_2(r;n)=[0,r_1] \cup (r_2, 1]$,
where $r_1=[m,1, 2]$ and $r_2=[m,2, 2n-2]$, a contradiction to the hypothesis.

\medskip
\noindent {\bf Case 3.} Either both $k=2$ and $m_2 \ge 3$ or both $k \ge 3$ and $m_2 \ge 2$.
\medskip

In this case, by \cite[Corollary~3.14(2)]{lee_sakuma_9},
$(m,m)$ appears in $S_2$, so in $CS(s)$, as a subsequence.
So $l_2\ge 2$ by \cite[Lemma~3.5]{lee_sakuma_9}, and thus we have
\[
\tilde{r}=[m_2-1,m_3, \dots, m_k] \quad \text{\rm and} \quad
\tilde{s}=[l_2-1,l_3, \dots, l_t].
\]
It follows from $s \in I_1(r;n) \cup I_2(r;n)$ that
$\tilde{s} \in I_1(\tilde{r};n) \cup I_2(\tilde{r};n)$.
At this point, we consider three subcases separately.

\medskip
\noindent {\bf Case 3.a.} $k=2$ and $m_2 \ge 3$.
\medskip

By \cite[Lemma~3.12(3)]{lee_sakuma_9}, $S_1=(m+1)$ and $S_2 =((m_2-1)\langle m \rangle)$.
Since $((2n-2) \langle S_1, S_2 \rangle, S_1)$ is contained in $CS(s)$ by the assumption,
$CS(\tilde{s})=CT(s)$ contains $((2n-2) \langle m_2-1 \rangle)$ as a subsequence.
But since $\tilde{r}=1/(m_2-1)=[m_2-1]$ and
$\tilde{s} \in I_1(\tilde{r};n) \cup I_2(\tilde{r};n)$,
this gives a contradiction to \cite[Lemma~5.1]{lee_sakuma_9}.

\medskip
\noindent {\bf Case 3.b.} $k \ge 3$ is even and $m_2 \ge 2$.
\medskip

Let $S(\tilde{r})= (T_1, T_2, T_1, T_2)$ be
the decomposition of $S(\tilde{r})$ given by \cite[Lemma~3.9]{lee_sakuma_9}.
Since $S_1$ begins and ends with $m+1$, $S_2$ begins and ends with $m$,
and since $((2n-2) \langle S_1, S_2 \rangle, S_1)$
is contained in $CS(s)$ by the assumption,
we see by \cite[Lemma~3.12(4)]{lee_sakuma_9} that
$CS(\tilde{s})=CT(s)$ contains $((2n-2) \langle T_2, T_1 \rangle, T_2)$ as a subsequence.
But since $\tilde{r}=[m_2-1, m_3, \dots, m_k]$ and
$\tilde{s} \in I_1(\tilde{r};n) \cup I_2(\tilde{r};n)$,
this gives a contradiction to the induction hypothesis.

\medskip
\noindent {\bf Case 3.c.} $k \ge 3$ is odd and $m_2 \ge 2$.
\medskip

Let $S(\tilde{r})= (T_1, T_2, T_1, T_2)$ be
the decomposition of $S(\tilde{r})$ given by \cite[Lemma~3.9]{lee_sakuma_9}.
Since $S_1$ begins and ends with $m+1$, $S_2$ begins and ends with $m$,
and since $((2n-2) \langle S_2, S_1 \rangle, S_2)$ is contained in $CS(s)$ by the assumption,
we see by \cite[Lemma~3.12(4)]{lee_sakuma_9} that
$CS(\tilde{s})=CT(s)$ contains, as a subsequence,
\[
(t_1+\ell', t_2, \dots, t_{s_1-1}, t_{s_1}, T_2, (2n-3) \langle T_1, T_2 \rangle, t_{s_1-1}, t_{s_1}+\ell''),
\]
where $(t_1, t_2, \dots,t_{s_1})=T_1$ and $\ell', \ell'' \in\ZZ_+\cup\{0\}$.
Since $t_1=t_{s_1}=(m_2-1)+1=m_2$ by \cite[Lemma~3.9]{lee_sakuma_9},
this actually implies that $\ell'=\ell''=0$,
and therefore
$CS(\tilde{s})$ contains
$((2n-2) \langle T_1, T_2 \rangle, T_1)$ as a subsequence.
But since $\tilde{r}=[m_2-1, m_3, \dots, m_k]$ and
$\tilde{s} \in I_1(\tilde{r};n) \cup I_2(\tilde{r};n)$,
this gives a contradiction to the induction hypothesis.

The proof of Lemma~\ref{lem:connection} is now completed.
\end{proof}

As an easy consequence of Lemma~\ref{lem:connection} and
\cite[Lemma~4.3(3)]{lee_sakuma_9},
we obtain the following.

\begin{corollary}
\label{cor:consecutive_vertices}
For any rational number $s \in I_1(r;n)\cup I_2(r;n)$,
the cyclic word $(u_s)$ cannot contain a subword
$w$ of the cyclic word $(u_r^{\pm n})$
which is a product of $4n-1$ pieces
but is not a product of less than $4n-1$ pieces.
\end{corollary}

\subsection{The case when $s \in I_1(r) \cup I_2(r)$}

If $\RGP{r}$ is the group of automorphisms of
the Farey tessellation $\DD$ generated by reflections in the edges of $\DD$ with an endpoint $r$,
and $\RGPP{r}$ is the group generated by $\RGP{r}$ and $\RGP{\infty}$,
then the region, $Q$, bounded by a pair of Farey edges with an endpoint $\infty$
and a pair of Farey edges with an endpoint $r$
forms a fundamental domain of the action of $\RGPP{r}$ on $\HH^2$.
Let $I_1(r)$ and $I_2(r)$ be the closed intervals in $\RRR$
obtained as the intersection with $\RRR$ of the closure of $Q$.
Then the intervals $I_1(r)$ and $I_2(r)$ are given by
$I_1(r)=[0,\hat{r}_1]$ and $I_2(r)=[\hat{r}_2,1]$, where
\begin{align*}
\hat{r}_1 &=
\begin{cases}
[m_1, m_2, \dots, m_{k-1}, m_k-1] & \mbox{if $k$ is even,}\\
[m_1, m_2, \dots, m_{k-1}] & \mbox{if $k$ is odd,}
\end{cases}\\
\hat{r}_2 &=
\begin{cases}
[m_1, m_2, \dots, m_{k-1}] & \mbox{if $k$ is even,}\\
[m_1, m_2, \dots, m_{k-1}, m_k-1] & \mbox{if $k$ is odd.}
\end{cases}
\end{align*}
Clearly $I_1(r) \subsetneq I_1(r;n)$ and $I_2(r) \subsetneq I_2(r;n)$.
It was shown in {\cite[Proposition~4.6]{Ohtsuki-Riley-Sakuma}} that
if two elements $s$ and $s'$ of $\QQQ$ belong to the same $\RGPP{r}$-orbit,
then the unoriented loops $\alpha_s$ and $\alpha_{s'}$ are homotopic in $S^3-K(r)$.

\begin{lemma}
\label{lem:inside_orbit}
Let $S(r)=(S_1, S_2, S_1, S_2)$ be as in \cite[{\it Lemma}~3.9]{lee_sakuma_9}.
For any rational number $s \in I_1(r) \cup I_2(r)$,
either $S_1$ or $S_2$
cannot occur in $CS(s)$ as a subsequence.
\end{lemma}

\begin{proof}
The assertion for the case when $s\ne 0$ is nothing other than
\cite[Proposition~3.19(1)]{lee_sakuma_2},
while the assertion for the case $s=0$ follows from the fact that
$CS(u_0)=\lp 2 \rp$ (see \cite[Remark~3.2]{lee_sakuma_9}).
\end{proof}

\subsection{The case when $s \in I_1(r;n) \setminus I_1(r)$
provided $k$ is even, and $s \in I_2(r;n) \setminus I_2(r)$
provided $k$ is odd}

In this subsection, we investigate an important property of $CS(s)$
for a rational number $s$ such that
\[
\begin{cases}
s \in I_1(r;n) \setminus I_1(r) & \text{if $k$ is even};\\
s \in I_2(r;n) \setminus I_2(r) & \text{if $k$ is odd}.
\end{cases}
\]

\begin{lemma}
\label{lem:outside_orbit2}
Let $S(r)=(S_1, S_2, S_1, S_2)$ be as in \cite[{\it Lemma}~3.9]{lee_sakuma_9}.
\begin{enumerate}[\indent \rm (1)]
\item If $k$ is even and $[m_1, \dots, m_k-1] < s \le [m_1, \dots, m_k-1,2]$,
then $CS(s)$ contains $(m+1, S_{2e}, S_1, S_2, S_1, S_{2b}, m+1)$ as a subsequence,
where $(m, S_{2e})=(S_{2b}, m)=S_2$.

\item If $k$ is odd and $[m_1, \dots, m_k-1,2] \le s < [m_1, \dots, m_k-1]$,
then $CS(s)$ contains $(m, S_{1e}, S_2, S_1, S_2, S_{1b}, m)$ as a subsequence,
where $(m+1, S_{1e})=(S_{1b}, m+1)=S_1$.
\end{enumerate}
\end{lemma}

\begin{proof}
We prove (1) and (2) simultaneously by induction on $k \ge 2$.
Let $s$ satisfy
\[
\begin{cases}
[m_1, \dots, m_k-1] < s \le [m_1, \dots, m_k-1,2] & \textrm{if $k$ is even;}\\
[m_1, \dots, m_k-1,2] \le s < [m_1, \dots, m_k-1] & \textrm{if $k$ is odd.}
\end{cases}
\]
Write $s$ as a continued fraction expansion $s=[l_1, \dots, l_t]$,
where $t \ge 1$, $(l_1, \dots, l_t) \in (\ZZ_+)^t$ and $l_t \ge 2$.
Then we have
$t \ge k+1$, $l_1=m_1, \dots, l_{k-1}=m_{k-1}, l_k=m_k-1$ and $l_{k+1}\ge 2$.

Throughout the proof, denote by $\tilde{r}$ and $\tilde{s}$
the rational numbers defined as in \cite[Lemma~3.8]{lee_sakuma_9}
for the rational numbers $r$ and $s$,
so that $CS(\tilde{r})=CT(r)$ and $CS(\tilde{s})=CT(s)$.

We consider three cases separately.

\medskip
\noindent {\bf Case 1.} $m_2=1$.
\medskip

In this case, $k \ge 3$, $l_2=m_2=1$ and $t \ge 4$.
So we have
\[
\tilde{r}=[m_3, \dots, m_k] \quad \text{\rm and} \quad
\tilde{s}=[l_3, \dots, l_t].
\]
It follows from the assumption that
\[
\begin{cases}
[m_3, \dots, m_k-1] < \tilde s \le [m_3, \dots, m_k-1,2] & \textrm{if $k$ is even;}\\
[m_3, \dots, m_k-1,2] \le \tilde s < [m_3, \dots, m_k-1] & \textrm{if $k$ is odd.}
\end{cases}
\]
This enables us to use the induction hypothesis.
At this point, we divide this case into three subcases.

\medskip
\noindent {\bf Case 1.a.} $k=3$.
\medskip

Since $[m_3-1,2] \le \tilde{s} < [m_3-1]$,
we see by \cite[Lemma~5.5]{lee_sakuma_9} that
$CS(\tilde{s})$ contains $(m_3-1, m_3, m_3-1)$ as a subsequence.
Since $CT(s)=CS(\tilde{s})$, this implies that
$CS(s)$ contains a subsequence
\[
(m, (m_3-1) \langle m+1 \rangle, m,
m_3 \langle m+1 \rangle, m, (m_3-1) \langle m+1 \rangle, m).
\]
Since $S_1 =(m_3\langle m+1 \rangle)$ and $S_2 =(m)$ by \cite[Lemma~3.12(1)]{lee_sakuma_9},
$CS(s)$ contains $(m, S_{1e}, S_2, S_1, S_2, S_{1b}, m)$ as a subsequence,
where $(m+1, S_{1e})=(S_{1b}, m+1)=S_1$.
So the assertion holds.

\medskip
\noindent {\bf Case 1.b.} $k\ge 4$ is even.
\medskip

Let $S(\tilde{r})=(T_1, T_2, T_1, T_2)$ be the decomposition of $S(\tilde{r})$
given by \cite[Lemma~3.9]{lee_sakuma_9}.
Since $[m_3, \dots, m_k-1] < \tilde{s} \le [m_3, \dots, m_k-1,2]$,
by the induction hypothesis
$CS(\tilde{s})$ contains
$(m_3+1, T_{2e}, T_1, T_2, T_1, T_{2b}, m_3+1)$ as a subsequence,
where $(m_3, T_{2e})=(T_{2b}, m_3)=T_2$.
Since $CS(\tilde{s})=CT(s)$, we see
by using \cite[Lemma~3.12(2)]{lee_sakuma_9}
that $CS(s)$ contains
$(m+1, S_{2e}, S_1, S_2, S_1, S_{2b}, m+1)$ as a subsequence,
where $(m, S_{2e})=(S_{2b}, m)=S_2$.

\medskip
\noindent {\bf Case 1.c.} $k\ge 4$ is odd.
\medskip

Let $S(\tilde{r})=(T_1, T_2, T_1, T_2)$ be the decomposition of $S(\tilde{r})$
given by \cite[Lemma~3.9]{lee_sakuma_9}.
Then, by the induction hypothesis,
$CS(\tilde{s})$ contains
$(m_3, T_{1e}, T_2, T_1, T_2, T_{1b}, m_3)$ as a subsequence,
where $(m_3+1, T_{1e})=(T_{1b}, m_3+1)=T_1$.
Since $CS(\tilde{s})=CT(s)$, we see
by using \cite[Lemma~3.12(2)]{lee_sakuma_9} that
$CS(s)$ contains
$(m, S_{1e}, S_2, S_1, S_2, S_{1b}, m)$ as a subsequence,
where $(m+1, S_{1e})=(S_{1b}, m+1)=S_1$.

\medskip
\noindent {\bf Case 2.} $k=2$ and $m_2=2$.
\medskip

In this case, $r=[m,2]$ and $[m+1] < s \le [m, 1, 2]$.
Then for $s=[l_1, \dots, l_t]$, we have $t \ge 3$, $l_1=m$, $l_2=1$ and $l_3 \ge 2$,
so that $\tilde{s}=[l_3,\dots,l_t]$ with $t \ge 3$ and $l_3 \ge 2$.
Hence $CS(\tilde{s})=CT(s)$ contains $(l_3, l_3)$ or $(l_3, l_3+1)$ as a subsequence.
Since $l_3 \ge 2$, this implies that $CS(s)$ contains $(m+1, m+1, m, m+1, m+1)$ as a subsequence.
Since $S_1=(m+1)$ and $S_2=(m)$ by \cite[Lemma~3.12(3)]{lee_sakuma_9}
and hence $S_{2e}=S_{2b}=\emptyset$,
$CS(s)$ contains a subsequence
$(m+1, S_{2e}, S_1, S_2, S_1, S_{2b}, m+1)$.
So the assertion holds.

\medskip
\noindent {\bf Case 3.} Either both $k=2$ and $m_2 \ge 3$ or both $k \ge 3$ and $m_2 \ge 2$.
\medskip

In this case, $l_2 \ge 2$.
So we have
\[
\tilde{r}=[m_2-1, \dots, m_k] \quad \text{\rm and} \quad
\tilde{s}=[l_2-1, \dots, l_t].
\]
It follows from the assumption that
\[
\begin{cases}
[m_2-1, m_3, \dots, m_k-1, 2] \le \tilde{s} < [m_2-1, m_3, \dots, m_k-1] & \textrm{if $k$ is even};\\
[m_2-1, m_3, \dots, m_k-1] < \tilde{s} \le [m_2-1, m_3, \dots, m_k-1, 2] & \textrm{if $k$ is odd}.
\end{cases}
\]
This enables us to use the induction hypothesis.
At this point, we divide this case into three subcases.

\medskip
\noindent {\bf Case 3.a.} $k=2$ and $m_2 \ge 3$.
\medskip

In this case, $\tilde{r}=[m_2-1]$.
Since $[m_2-2,2] \le \tilde{s} < [m_2-2]$,
we see by \cite[Lemma~5.5]{lee_sakuma_9} that
$CS(\tilde{s})$ contains $(m_2-2, m_2-1, m_2-2)$ as a subsequence.
Since $CT(s)=CS(\tilde{s})$, this implies that
$CS(s)$ contains
\[
(m+1, (m_2-2) \langle m \rangle, m+1, (m_2-1) \langle m \rangle, m+1, (m_2-2) \langle m \rangle, m+1),
\]
as a subsequence.
Since $S_1=(m+1)$ and $S_2=((m_2-1)\langle m \rangle)$ by \cite[Lemma~3.12(3)]{lee_sakuma_9},
$CS(s)$ contains a subsequence
$(m+1, S_{2e}, S_1, S_2, S_1, S_{2b}, m+1)$,
where $(m, S_{2e})=(S_{2b}, m)=S_2$.
Hence the assertion holds.

\medskip
\noindent {\bf Case 3.b.} $k \ge 3$ is even and $m_2 \ge 2$.
\medskip

Let $S(\tilde{r})= (T_1, T_2, T_1, T_2)$ be
the decomposition of $S(\tilde{r})$ given by \cite[Lemma~3.9]{lee_sakuma_9}.
Then, by the induction hypothesis,
$CS(\tilde{s})$ contains
$(m_2-1, T_{1e}, T_2, T_1, T_2, T_{1b}, m_2-1)$ as a subsequence,
where $(m_2, T_{1e})=(T_{1b}, m_2)=T_1$.
Since $CS(\tilde{s})=CT(s)$,
we see by using \cite[Lemma~3.12(4)]{lee_sakuma_9} that
$CS(s)$ contains $(m+1, S_{2e}, S_1, S_2, S_1, S_{2b}, m+1)$ as a subsequence,
where $(m, S_{2e})=(S_{2b}, m)=S_2$.

\medskip
\noindent {\bf Case 3.c.} $k \ge 3$ is odd and $m_2 \ge 2$.
\medskip

Let $S(\tilde{r})=(T_1, T_2, T_1, T_2)$ be
the decomposition of $S(\tilde{r})$ given by \cite[Lemma~3.9]{lee_sakuma_9}.
Then, by the induction hypothesis,
$CS(\tilde{s})$ contains
$(m_2, T_{2e}, T_1, T_2, T_1, T_{2b}, m_2)$ as a subsequence,
where $(m_2-1, T_{2e})=(T_{2b}, m_2-1)=T_2$.
Since $CS(\tilde{s})=CT(s)$, we see
by using \cite[Lemma~3.12(4)]{lee_sakuma_9} that
$CS(s)$ contains
$(m, S_{1e}, S_2, S_1, S_2, S_{1b}, m)$ as a subsequence,
where $(m+1, S_{1e})=(S_{1b}, m+1)=S_1$.

The proof of Lemma~\ref{lem:outside_orbit2} is now completed.
\end{proof}

\subsection{The case when $s \in I_2(r;n) \setminus I_2(r)$
provided $k$ is even, and $s \in I_1(r;n) \setminus I_1(r)$
provided $k$ is odd}

Finally, we investigate an important property of $CS(s)$
for a rational number $s$ such that
\[
\begin{cases}
s \in I_2(r;n) \setminus I_2(r) & \text{if $k$ is even};\\
s \in I_1(r;n) \setminus I_1(r) & \text{if $k$ is odd}.
\end{cases}
\]

\begin{lemma}
\label{lem:outside_orbit}
Let $S(r)=(S_1, S_2, S_1, S_2)$ be as in \cite[{\it Lemma}~3.9]{lee_sakuma_9}.
\begin{enumerate}[\indent \rm (1)]
\item If $k$ is even and $[m_1, \dots, m_k, 2n-2] < s < [m_1, \dots, m_{k-1}]$,
then $CS(s)$ contains $(m, S_{1e}, d \langle S_2, S_1 \rangle, S_2, S_{1b}, m)$ as a subsequence,
where $1 \le d \le 2n-3$ and $(m+1, S_{1e})=(S_{1b}, m+1)=S_1$.

\item If $k$ is odd and $[m_1, \dots, m_{k-1}] < s < [m_1, \dots, m_k, 2n-2]$,
then $CS(s)$ contains $(m+1, S_{2e}, d \langle S_1, S_2 \rangle, S_1, S_{2b}, m+1)$ as a subsequence,
where $1 \le d \le 2n-3$ and $(m, S_{2e})=(S_{2b}, m)=S_2$.
\end{enumerate}
\end{lemma}

\begin{proof}
We prove (1) and (2) simultaneously by induction on $k \ge 2$.
Let $s$ satisfy
\[
\begin{cases}
[m_1, \dots, m_k, 2n-2] < s < [m_1, \dots, m_{k-1}] & \textrm{if $k$ is even;}\\
[m_1, \dots, m_{k-1}] < s < [m_1, \dots, m_k, 2n-2] & \textrm{if $k$ is odd.}
\end{cases}
\]
Write $s$ as a continued fraction expansion $s=[l_1, \dots, l_t]$, where
$(l_1, \dots, l_t) \in (\ZZ_+)^t$ and $l_t \ge 2$.
Then $l_1=m$.

Throughout the proof, denote by $\tilde{r}$ and $\tilde{s}$
the rational numbers defined as in \cite[Lemma~3.8]{lee_sakuma_9}
for the rational numbers $r$ and $s$,
so that $CS(\tilde{r})=CT(r)$ and $CS(\tilde{s})=CT(s)$.

We consider three cases separately.

\medskip
\noindent {\bf Case 1.} $m_2=1$.
\medskip

In this case, $k \ge 3$, $l_2=m_2=1$ and $t \ge 3$.
So we have
\[
\tilde{r}=[m_3, \dots, m_k] \quad \text{\rm and} \quad
\tilde{s}=[l_3, \dots, l_t].
\]
It follows from the assumption that
\[
\begin{cases}
[m_3, \dots, m_k, 2n-2] < \tilde{s} < [m_3, \dots, m_{k-1}] & \textrm{if $k$ is even;}\\
[m_3, \dots, m_{k-1}] < \tilde{s} < [m_3, \dots, m_k, 2n-2] & \textrm{if $k$ is odd.}
\end{cases}
\]
This enables us to use the induction hypothesis.
At this point, we divide this case into three subcases.

\medskip
\noindent {\bf Case 1.a.} $k=3$.
\medskip

Since $\tilde r=[m_3]$ and $0 < \tilde{s} < [m_3, 2n-2]$,
we see by \cite[Lemma~5.4]{lee_sakuma_9} that
$CS(\tilde{s})$ contains a subsequence
$(m_3+c, d \langle m_3 \rangle, m_3+c')$
for some $c, c'\ge 1$ and $0 \le d \le 2n-4$.
Since $CT(s)=CS(\tilde{s})$, this implies that
$CS(s)$ contains a subsequence
\[
((m_3+c)\langle m+1 \rangle, m, d \langle m_3 \langle m+1 \rangle, m \rangle,
(m_3+c') \langle m+1 \rangle),
\]
where $0 \le d \le 2n-4$.
In particular, $CS(s)$ contains a subsequence
\[
(m+1, m_3 \langle m+1 \rangle, m, d \langle m_3 \langle m+1 \rangle, m \rangle, m_3 \langle m+1 \rangle, m+1).
\]
Since $S_1 =(m_3\langle m+1 \rangle)$ and $S_2 =(m)$ by \cite[Lemma~3.12(1)]{lee_sakuma_9},
$CS(s)$ contains a subsequence
$(m+1, d' \langle S_1, S_2 \rangle, S_1, m+1)$,
where $d'=d+1$.
This implies the assertion, because
$S_2=(m)$ and $S_{2e}=S_{2b}=\emptyset$.

\medskip
\noindent {\bf Case 1.b.} $k\ge 4$ is even.
\medskip

Let $S(\tilde{r})=(T_1, T_2, T_1, T_2)$ be the decomposition of $S(\tilde{r})$
given by \cite[Lemma~3.9]{lee_sakuma_9}.
Then, by the induction hypothesis,
$CS(\tilde{s})$ contains
$(m_3, T_{1e}, d \langle T_2, T_1 \rangle, T_2, T_{1b}, m_3)$ as a subsequence,
where $1 \le d \le 2n-3$ and $(m_3+1, T_{1e})=(T_{1b}, m_3+1)=T_1$.
Since $CS(\tilde{s})=CT(s)$, we see by using \cite[Lemma~3.12(2)]{lee_sakuma_9}
that $CS(s)$ contains
$(m, S_{1e}, d \langle S_2, S_1 \rangle, S_2, S_{1b}, m)$ as a subsequence,
where $1 \le d \le 2n-3$ and $(m+1, S_{1e})=(S_{1b}, m+1)=S_1$.
In fact, we have the following identity under the notation of \cite[Lemma~3.12(2)]{lee_sakuma_9}:
\[
\begin{aligned}
(m, S_{1e})
&=(m, (t_1-1) \langle m+1 \rangle,  m, t_2 \langle m+1 \rangle,
\dots, t_{s_1-1}\langle m+1 \rangle, m, t_{s_1}\langle m+1 \rangle)\\
&=(m, m_3 \langle m+1 \rangle,  m, t_2 \langle m+1 \rangle,
\dots, t_{s_1-1}\langle m+1 \rangle, m, t_{s_1}\langle m+1 \rangle).
\end{aligned}
\]
Thus the ``$T$-sequence'' of $(m, S_{1e})$ is $(m_3, T_{1e})$.
Similarly, the ``$T$-sequence'' of $(S_{1b}, m)$ is $(T_{1b},m_3)$.
By using these facts, we can confirm the assertion above.

\medskip
\noindent {\bf Case 1.c.} $k\ge 4$ is odd.
\medskip

Let $S(\tilde{r})=(T_1, T_2, T_1, T_2)$ be the decomposition of $S(\tilde{r})$
given by \cite[Lemma~3.9]{lee_sakuma_9}.
Then, by the induction hypothesis,
$CS(\tilde{s})$ contains
$(m_3+1, T_{2e}, d \langle T_1, T_2 \rangle, T_1, T_{2b}, m_3+1)$ as a subsequence,
where $1 \le d \le 2n-3$ and $(m_3, T_{2e})=(T_{2b}, m_3)=T_2$.
Since $CS(\tilde{s})=CT(s)$, we see
by using \cite[Lemma~3.12(2)]{lee_sakuma_9} that
$CS(s)$ contains
$(m+1, S_{2e}, d \langle S_1, S_2 \rangle, S_1, S_{2b}, m+1)$ as a subsequence,
where $1 \le d \le 2n-3$ and $(m, S_{2e})=(S_{2b}, m)=S_2$.

\medskip
\noindent {\bf Case 2.} $k=2$ and $m_2=2$.
\medskip

In this case, $r=[m,2]$ and $[m, 2, 2n-2] < s < [m]$.
Then one of the following holds for $s=[l_1, \dots, l_t]$.

\begin{enumerate}[\indent \rm (i)]
\item $t \ge 3$, $l_1=m$, $l_2=2$ and $l_3 \le 2n-3$; or

\item $t \ge 2$, $l_1=m$ and $l_2 \ge 3$.
\end{enumerate}
If (i) happens,
we claim that $CS(s)$ contains a subsequence
$(2 \langle m \rangle, m+1, d \langle m, m+1 \rangle, 2 \langle m \rangle)$,
where $0 \le d \le 2n-4$.
Clearly ${\tilde s}=[1, l_3, \dots, l_t]$.
Here, if $l_3=1$, then $t \ge 4$ and $CS({\tilde s})=CT(s)$ contains
a subsequence $(2,2)$. So $CS(s)$ contains a subsequence
$(2 \langle m \rangle, m+1, 2 \langle m \rangle)$
and therefore the claim holds with $d=0$.
Also if $l_3 \ge 2$, then $CS({\tilde s})=CT(s)$ contains
a subsequence $(2,(l_3-1)\langle 1 \rangle,2)$,
so that $CS(s)$ contains a subsequence
$(2 \langle m \rangle, m+1, (l_3-1) \langle m, m+1 \rangle, 2 \langle m \rangle)$
and therefore the claim holds with $d=l_3-1 \le 2n-4$.
Then, since $S_1=(m+1)$ and $S_2=(m)$, $CS(s)$ contains
$(m, d' \langle S_2, S_1 \rangle, S_2, m)$ as a subsequence,
where $d'=d+1$. Since $S_1=(m+1)$, $S_{1e}=S_{1b}=\emptyset$.
Hence the assertion holds.

On the other hand, if (ii) happens,
we claim that $CS(s)$ contains a subsequence
$((l_2-1) \langle m \rangle, m+1, (l_2-1) \langle m \rangle)$,
where $l_2-1 \ge 2$.
Clearly ${\tilde s}=[l_2-1, l_3, \dots, l_t]$.
Here, if either $t=2$ or $l_3 \ge 2$, then $CS({\tilde s})=CT(s)$ contains
a subsequence $(l_2-1,l_2-1)$, so that $CS(s)$ contains a subsequence
$((l_2-1) \langle m \rangle, m+1, (l_2-1) \langle m \rangle)$, as desired.
Also if $l_3=1$, then $t \ge 4$ and $CS({\tilde s})=CT(s)$ contains
a subsequence $(l_2, l_2)$.
Then $CS(s)$ contains a subsequence
$(m, (l_2-1) \langle m \rangle, m+1, (l_2-1) \langle m \rangle, m)$,
and therefore $CS(s)$ contains a subsequence
$((l_2-1) \langle m \rangle, m+1, (l_2-1) \langle m \rangle)$,
as desired.
Then, since $S_1=(m+1)$ and $S_2=(m)$,
$CS(s)$ contains $(m, S_2, S_1, S_2, m)$ as a subsequence,
so the assertion holds.

\medskip
\noindent {\bf Case 3.} Either both $k=2$ and $m_2 \ge 3$ or both $k \ge 3$ and $m_2 \ge 2$.
\medskip

In this case, $l_2 \ge 2$.
So we have
\[
\tilde{r}=[m_2-1, \dots, m_k] \quad \text{\rm and} \quad
\tilde{s}=[l_2-1, \dots, l_t].
\]
It follows from the assumption that
\[
\begin{cases}
[m_2-1, \dots, m_{k-1}] < \tilde{s} < [m_2-1, \dots, m_k, 2n-2] & \textrm{if $k$ is even};\\
[m_2-1, \dots, m_k, 2n-2] < \tilde{s} < [m_2-1, \dots, m_{k-1}] & \textrm{if $k$ is odd}.
\end{cases}
\]
This enables us to use the induction hypothesis.
At this point, we divide this case into three subcases.

\medskip
\noindent {\bf Case 3.a.} $k=2$ and $m_2 \ge 3$.
\medskip

Then $\tilde r=[m_2-1]$ and $0 < \tilde{s} < [m_2-1, 2n-2]$.
Hence we see by \cite[Lemma~5.4]{lee_sakuma_9} that
$CS(\tilde{s})$ contains a subsequence
$(m_2-1+c, d \langle m_2-1 \rangle, m_2-1+c')$
for some $c, c'\ge 1$ and $0 \le d \le 2n-4$.
Since $CT(s)=CS(\tilde{s})$, this implies that
$CS(s)$ contains a subsequence
\[
(m, (m_2-1) \langle m \rangle, m+1, d \langle (m_2-1) \langle m \rangle, m+1\rangle, (m_2-1) \langle m \rangle, m),
\]
where $0 \le d \le 2n-4$.
Since $S_1=(m+1)$ and $S_2=((m_2-1)\langle m \rangle)$ by \cite[Lemma~3.12(3)]{lee_sakuma_9},
$CS(s)$ contains a subsequence
$(m, d' \langle S_2, S_1 \rangle, S_2, m)$,
where $d'=d+1$.
Since $S_1=(m+1)$ and therefore
$S_{1e}=S_{1b}=\emptyset$,
the assertion holds.

\medskip
\noindent {\bf Case 3.b.} $k \ge 3$ is even and $m_2 \ge 2$.
\medskip

Let $S(\tilde{r})= (T_1, T_2, T_1, T_2)$ be
the decomposition of $S(\tilde{r})$ given by \cite[Lemma~3.9]{lee_sakuma_9}.
Then, by the induction hypothesis,
$CS(\tilde{s})$ contains
$(m_2, T_{2e}, d \langle T_1, T_2 \rangle, T_1, T_{2b}, m_2)$ as a subsequence,
where $1 \le d \le 2n-3$ and $(m_2-1, T_{2e})=(T_{2b}, m_2-1)=T_2$.
Since $CS(\tilde{s})=CT(s)$, we see by using \cite[Lemma~3.12(4)]{lee_sakuma_9} that
$CS(s)$ contains
$(m, S_{1e}, d \langle S_2, S_1 \rangle, S_2, S_{1b}, m)$ as a subsequence,
where $1 \le d \le 2n-3$ and $(m+1, S_{1e})=(S_{1b}, m+1)=S_1$.

\medskip
\noindent {\bf Case 3.c.} $k \ge 3$ is odd and $m_2 \ge 2$.
\medskip

Let $S(\tilde{r})= (T_1, T_2, T_1, T_2)$ be
the decomposition of $S(\tilde{r})$ given by \cite[Lemma~3.9]{lee_sakuma_9}.
Then, by the induction hypothesis,
$CS(\tilde{s})$ contains
$(m_2-1, T_{1e}, d \langle T_2, T_1 \rangle, T_2, T_{1b}, m_2-1)$ as a subsequence,
where $1 \le d \le 2n-3$ and $(m_2, T_{1e})=(T_{1b}, m_2)=T_1$.
Since $CS(\tilde{s})=CT(s)$, we see
by using \cite[Lemma~3.12(4)]{lee_sakuma_9}
that $CS(s)$ contains
$(m+1, S_{2e}, d \langle S_1, S_2 \rangle, S_1, S_{2b}, m+1)$ as a subsequence,
where $1 \le d \le 2n-3$ and $(m, S_{2e})=(S_{2b}, m)=S_2$.

The proof of Lemma~\ref{lem:outside_orbit} is now completed.
\end{proof}

\section{Proof of Main Theorem~\ref{thm:conjugacy}(1)}
\label{sec:proof of main theorem(1) for the general case}
Consider a Heckoid group $\Hecke(r;n)$,
where $r$ is a non-integral rational number and $n$ is an integer greater than $1$.
By \cite[Lemma~2.5]{lee_sakuma_9}, we may assume $0< r\le 1/2$.
Since we have already treated, in \cite{lee_sakuma_9},
the case where $r=1/p$ for some integer $p\ge 2$,
we may assume
$r=[m_1, \dots, m_k]$ with $m_1=m \ge 2$ and $k \ge 2$.
Let $s$ and $s'$ be distinct rational numbers in $I_1(r;n)\cup I_2(r;n)$.
Suppose on the contrary that the simple loops
$\alpha_s$ and $\alpha_{s'}$ are homotopic in $\orbs(r;n)$, i.e.,
$u_s$ and $u_{s'}^{\pm 1}$ are conjugate in $\Hecke(r;n)$.
By \cite[Lemma~4.11]{lee_sakuma_9}, there is a reduced nontrivial annular
diagram $M$ over $\Hecke(r;n)=\langle a, b \svert u_r^n \rangle$ with
$(\phi(\alpha)) \equiv (u_s)$ and $(\phi(\delta)) \equiv (u_{s'}^{\pm 1})$,
where $\alpha$ and $\delta$ are, respectively, outer and inner boundary cycles of $M$.
Since $s, s' \in I_1(r;n)\cup I_2(r;n)$,
we see by Lemma~\ref{lem:connection}
that $CS(\phi(\alpha))$ and $CS(\phi(\delta))$
do not contain $((2n-1) \langle S_1, S_2 \rangle)$ nor $((2n-1) \langle S_2, S_1 \rangle)$
as a subsequence.
So by \cite[Corollary~4.17]{lee_sakuma_9}, $M$ is shaped as
in \cite[Figure~3(a)]{lee_sakuma_9} or \cite[Figure~3(b)]{lee_sakuma_9}.

\begin{lemma}
\label{lem:claim1_II}
$M$ is shaped as in \cite[{\it Figure}~3(a)]{lee_sakuma_9}, that is,
$M$ satisfies the conclusion of \cite[{\it Corollary}~4.17(1)]{lee_sakuma_9}.
\end{lemma}

\begin{proof}
Suppose on the contrary that $M$ is shaped as in \cite[Figure~3(b)]{lee_sakuma_9}.
Then $(\phi(\alpha)) \equiv (u_s)$ contains a subword of the cyclic word
$(u_r^{\pm n})$ which is a product of $4n-2$ pieces,
but is not a product of less than $4n-2$ pieces
(see \cite[Convention~4.7(3) and Theorem~4.15(4)]{lee_sakuma_9}).
Since $4n-2 \ge 6$, this
together with \cite[Lemma~4.2(2c)]{lee_sakuma_9}
implies that $CS(\phi(\alpha))=CS(s)$ contains both $S_1$ and $S_2$ as subsequences.
So by Lemma~\ref{lem:inside_orbit},
$s \notin I_1(r) \cup I_2(r)$.
Then by Lemmas~\ref{lem:outside_orbit2}
and \ref{lem:outside_orbit},
$(u_s)$ contains a subword $w$ for which $S(w)$ is a subsequence of $CS(s)$ such that
\[
S(w)=
\begin{cases}
(m+1, S_{2e}, S_1, S_2, S_1, S_{2b}, m+1) & \text{if $k$ is even and $s \in I_1(r;n)$};\\
(m, S_{1e}, d \langle S_2, S_1 \rangle, S_2, S_{1b}, m) & \text{if $k$ is even and $s \in I_2(r;n)$};\\
(m, S_{1e}, S_2, S_1, S_2, S_{1b}, m) & \text{if $k$ is odd and $s \in I_2(r;n)$};\\
(m+1, S_{2e}, d \langle S_1, S_2 \rangle, S_1, S_{2b}, m+1) & \text{if $k$ is odd and $s \in I_1(r;n)$},
\end{cases}
\]
where $1 \le d \le 2n-3$, $(m+1, S_{1e})=(S_{1b}, m+1)=S_1$
and $(m, S_{2e})=(S_{2b}, m)=S_2$.

\medskip
\noindent {\bf Claim.} {\it There is a face $D$ in the outer boundary layer of $M$ such that
$\phi(\partial D^+)$ is a subword of $w$.}

\begin{proof}[Proof of Claim]
Suppose that there is no such face.
Then either (i) there is a face, $D$, in the outer boundary layer of $M$ such that
$\phi(\partial D^+)\equiv uwv$ for some words $u$ and $v$
such that at least one of them is nonempty, or
(ii) there are two successive faces, say $D_1$ and $D_2$,
in the outer boundary layer of $M$
such that $\phi(\partial D^+_1)\equiv uw_1$ and
$\phi(\partial D^+_2)\equiv w_2v$,
where $u$, $v$, $w_1$ and $w_2$ are nonempty words such that
$w\equiv w_1w_2$.

First assume that (i) holds.
If $S(w)$ is of the first or the last form, namely,
if $S(w)$ begins and ends with $m+1$,
then $S(w)$ is a subsequence of $CS(\phi(\partial D))=\lp 2n\langle S_1,S_2\rangle \rp$.
By the uniqueness of $S_1$ in $CS(r)=\lp S_1,S_2,S_1,S_2 \rp$
(see \cite[Lemma~3.9]{lee_sakuma_9}), this implies that
the first $S_1$ in $S(w)$ must coincide some $S_1$ in $\lp 2n\langle S_1,S_2\rangle \rp$.
But, then this implies that $S(w)$ cannot be a subsequence of
$\lp 2n\langle S_1,S_2\rangle \rp$, because $(m+1, S_{2e})\ne S_2$,
a contradiction.
If $S(w)$ is of the second or the third form, namely,
if $S(w)$ begins and ends with $m$,
then, by the uniqueness of $S_1$ in $CS(r)=\lp S_1,S_2,S_1,S_2 \rp$,
we see that $S(\phi(\partial D^-))$
is equal to $(1, (2n-d-2)\langle S_2,S_1\rangle, S_2, 1)$
or $(1, (2n-3)\langle S_2,S_1\rangle, S_2, 1)$ accordingly.
Since both $2n-d-2$ and $2n-3$ are at least $1$,
we see by \cite[Lemma~4.2(2)]{lee_sakuma_9} that
the word $\phi(\partial D^-)$ cannot be expressed as a product of
$2$ pieces of $(u_r^{\pm n})$,
contradicting \cite[Figure~3(b)]{lee_sakuma_9}
(cf. \cite[Corollary~4.17(2)]{lee_sakuma_9}).

Next assume that (ii) holds.
If $S(w)$ is of the first form,
namely, if $S(w)=(m+1, S_{2e}, S_1, S_2, S_1, S_{2b}, m+1)$,
then either the first $S_1$ in $S(w)$ is a subsequence of
$S(\phi(\partial D^+_1))$ or the last $S_1$ in $S(w)$ is a subsequence of
$S(\phi(\partial D^+_2))$.
In either case, we encounter a contradiction by an argument as in (i).
The other three forms of $S(w)$ are treated similarly.
\end{proof}

For such a face $D$ as in the statement of the above claim,
since $CS(\phi(\partial D))=\lp 2n \langle S_1, S_2 \rangle \rp$,
$S(\phi(\partial D^-))$ must contain
$(S_1, S_2, \ell)$ as a subsequence for some $\ell \in \ZZ_+$.
In more detail, if $S(w)$ is of the first or fourth form,
then since $\phi(\partial D^+)$ is a subword of $w$,
$S(\phi(\partial D^-))$ must contain $(S_1, S_2, S_1)$ as a subsequence.
On the other hand, if $S(w)$ is of the second or third form,
then $S(\phi(\partial D^-))$ must contain
$(\ell_1, S_2, S_1, S_2, \ell_2)$
as a subsequence for some $\ell_1, \ell_2 \in \ZZ_+$.
But then by \cite[Lemma~4.2(2)]{lee_sakuma_9},
the word $\phi(\partial D^-)$ cannot be expressed as a product of
$2$ pieces of $(u_r^{\pm n})$,
contradicting \cite[Figure~3(b)]{lee_sakuma_9}
(cf. \cite[Corollary~4.17(2)]{lee_sakuma_9}).
\end{proof}

\begin{lemma}
\label{lem:claim2_II}
$s, s' \notin I_1(r) \cup I_2(r)$.
\end{lemma}

\begin{proof}
Suppose on the contrary that $s$ or $s'$ lies in $I_1(r) \cup I_2(r)$.
Without loss of generality, assume that $s \in I_1(r) \cup I_2(r)$.
By Lemma~\ref{lem:inside_orbit},
either $S_1$ or $S_2$ does not occur in $CS(s)$ as a subsequence.
But then by the feature of \cite[Figure~3(a)]{lee_sakuma_9},
$CS(\phi(\delta))=CS(s')$ contains both $S_1$ and $S_2$ as subsequences,
which implies by Lemma~\ref{lem:inside_orbit}
that $s' \notin I_1(r) \cup I_2(r)$
and therefore $s'\in (I_1(r;n)\backslash I_1(r))\cup (I_2(r;n)\backslash I_2(r))$.
By Lemmas~\ref{lem:outside_orbit2}
and \ref{lem:outside_orbit},
$(u_{s'})$ contains a subword $w$ for which $S(w)$ is a subsequence of $CS(s')$
of the form as in the proof of Lemma~\ref{lem:claim1_II}.
Thus, by the argument in the proof,
we see that
$CS(s)=CS(\phi(\alpha))$ contains $(S_1, S_2, \ell)$ as a subsequence for some $\ell \in \ZZ_+$, a contradiction.
\end{proof}

\begin{lemma}
\label{lem:terms_II}
Both $CS(s)$ and $CS(s')$ consist of $m$ and $m+1$.
\end{lemma}

\begin{proof}
By Lemma~\ref{lem:claim2_II} together with Lemmas~\ref{lem:outside_orbit2}
and \ref{lem:outside_orbit},
$CS(s)$ and $CS(s')$ contain both $S_1$ and $S_2$.
Hence by \cite[Lemmas~3.5 and 3.9]{lee_sakuma_9},
both $CS(s)$ and $CS(s')$ consist of $m$ and $m+1$.
\end{proof}

At this point, we introduce the concept for a vertex of $M$ to be
converging, diverging or mixing (cf. \cite[Section~7]{lee_sakuma_4}).
To this end, we subdivide the edges of $M$ so that the label of
any oriented edge in the subdivision has length $1$.
We call each of the edges in the subdivision a {\it unit segment}
in order to distinguish them from the edges in the original $M$.

\begin{definition}
\label{def:vertex_type}
{\rm
(1) A vertex $x$ in $M$ is said to be {\it converging} (resp., {\it diverging})
if the set of labels of incoming
unit segments of $x$ is $\{a, b\}$ (resp., $\{a^{-1}, b^{-1}\}$).
See Figure~\ref{fig.converging} and its caption for description.

(2) A vertex $x$ in $M$ is said to be {\it mixing}
if the set of labels of incoming
unit segments of $x$ is $\{a, a^{-1},b, b^{-1}\}$.
See Figure~\ref{fig.impossible_tsequence} and its caption for description.
}
\end{definition}

\begin{figure}[h]
\includegraphics{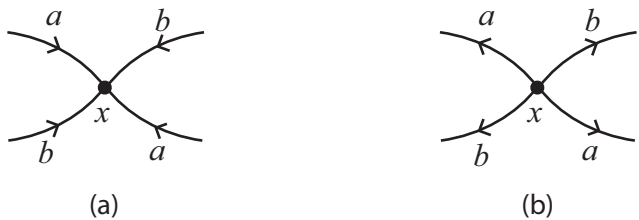}
\caption{
Orient each of the unit segment so that
the associated label is equal to $a$ or $b$.
Then a vertex $x$ is (a) converging (resp., (b) diverging)
if all unit segments incident on $x$
are oriented so that they are converging into $x$
(resp., diverging from $x$).}
\label{fig.converging}
\end{figure}

\begin{figure}[h]
\includegraphics{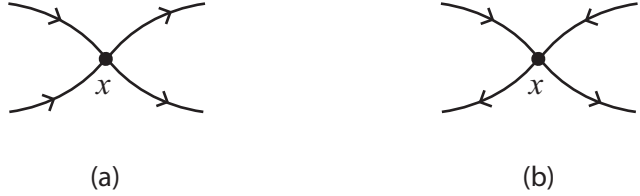}
\caption{
A vertex $x$ is mixing if it looks like as in the above when
we orient the segments as in Convention~\ref{con:figure} below.}
\label{fig.impossible_tsequence}
\end{figure}

The proof of Lemma~\ref{lem:converging_diverging} below
is a slight modification of
that of \cite[Proposition~7.5(2)]{lee_sakuma_4},
where we employ the following convention.

\begin{convention}
\label{con:figure}
{\rm
In Figures~\ref{fig.vertex_type}--\ref{fig.transformation_1},
the change of directions of consecutive arrowheads
represents the change from positive (negative, resp.) words
to negative (positive, resp.) words, and
a dot represents a vertex whose position is clearly identified.
Also small letters $c_i$ and $d_i$ ($i=1,2$) represent the lengths of the corresponding positive
(or negative) words.
The upper complementary region is regarded as the unbounded region
of $\RR^2-M$.
Thus the outer boundary cycles runs
the upper boundary from left to right.
}
\end{convention}

\begin{lemma}
\label{lem:converging_diverging}
We may assume that every vertex $x$ of $M$ with degree $4$
is either converging or diverging.
To be precise, we can modify the reduced nontrivial annular diagram $M$
into a reduced nontrivial annular diagram $M'$
keeping the outer and inner boundary labels unchanged
so that every vertex of $M'$ with degree $4$
is either converging or diverging.
In particular,
under \cite[{\it Notation}~4.18]{lee_sakuma_9},
every $S(\phi(\partial D_i^+))$ is a subsequence of both
$CS(\phi(\alpha))$ and $CS(\phi(\partial D_i))$.
Similarly, every $S(\phi(\partial D_i^-))$ is a subsequence of both
$CS(\phi(\delta^{-1}))$ and $CS(\phi(\partial D_i)^{-1})$.
\end{lemma}

\begin{proof}
Suppose on the contrary that there is a vertex $x \in M$ with degree $4$ such that
$x$ is neither converging nor diverging.
We may assume $x$ is the vertex between $D_1$ and $D_2$.
Then $x$ has one of the five types as depicted in Figure~\ref{fig.vertex_type},
where $c_i$ and $d_i$ ($i=1,2$) are positive integers,
up to simultaneous reversal of the edge orientations
and up to the reflection in the vertical edge passing through the vertex $x$.
To see this, let $L$ be the set of labels of incoming
unit segments of $x$, and
orient each of the unit segment so that
the associated label is equal to $a$ or $b$
as in Figure~\ref{fig.converging}.
If $L=\{a^{\pm 1}, b^{\pm 1}\}$, then
we obtain the situation (a) or (b) in Figure~\ref{fig.vertex_type}.
If $L$ consists of three elements,
then we may assume that $a$ and $a^{-1}$, respectively,
appear as the label of the upper left and lower right
incoming unit segments
and that $b$ or $b^{-1}$ does not belong to $L$.
Then we obtain the situation (c) or (d) in Figure~\ref{fig.vertex_type}.
If $L$ consists of two elements,
then we may assume both the upper left and lower right
incoming unit segments have label $a$,
and both the upper left and lower right
incoming unit segments have label $b^{-1}$,
because $x$ is not converging nor diverging.
In this case, we have the situation (e) in Figure~\ref{fig.vertex_type}.

\begin{figure}[h]
\includegraphics{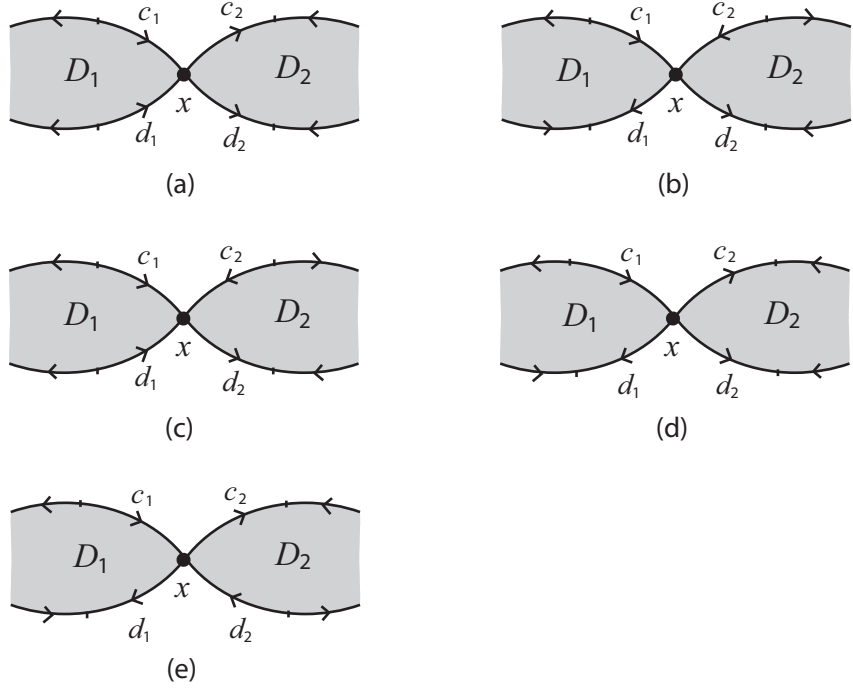}
\caption{
The five possible types of a vertex $x \in M$ with degree $4$
such that $x$ is neither converging nor diverging}
\label{fig.vertex_type}
\end{figure}

Assume that $x$ is depicted as in Figure~\ref{fig.vertex_type}(a).
Then, for each $i=1,2$, $c_i$ is a term of $CS(\phi(\partial D_i))=CS(u_r^n)$
and so is equal to $m$ or $m+1$.
Hence the term, $c_1+c_2$, of $CS(\phi(\alpha))=CS(s)$ is at least $2m$.
By Lemma~\ref{lem:terms_II}, $CS(s)$ consists of $m$ and $m+1$.
But since $2m > m+1$, we obtain a contradiction.

Assume that $x$ is depicted as in Figure~\ref{fig.vertex_type}(b).
Then $(c_1,c_2)$ is a subsequence of $CS(\phi(\alpha))=CS(s)$.
Since $CS(\phi(\alpha))=CS(s)$ consists of $m$ and $m+1$ by Lemma~\ref{lem:terms_II},
the only possibility is that
$c_1=c_2=m$ and $d_1=d_2=1$.
But then there is a term $1$ in $CS(s')$,
which is a contradiction, because $CS(s')$
also consists of $m$ and $m+1$ again by Lemma~\ref{lem:terms_II}.

Assume that $x$ is depicted as in Figure~\ref{fig.vertex_type}(c).
Then $(c_1,c_2)$ is a subsequence of $CS(\phi(\alpha))=CS(s)$
and $d_1+d_2$ is a term of $CS(\phi(\delta))=CS(s')$.
Thus each of $c_1$, $c_2$ and $d_1+d_2$ is either $m$ or $m+1$ by Lemma~\ref{lem:terms_II}.
Moreover, $c_2+d_2$ is a term of $CS(u_r^n)$
and hence it is either $m$ or $m+1$.
So, we have the following two possibilities:
\begin{enumerate}[\indent \rm (i)]
\item $c_1=m$, $c_2=m$, $d_1=m$, $d_2=1$;

\item $c_1=m+1$, $c_2=m$, $d_1=m$, $d_2=1$.
\end{enumerate}
In either case, since $c_2=d_1$,
we can transform $M$ so that $x$ is diverging
as in Figure~\ref{fig.transformation_2}.
To be precise, we cut $M$ at the black vertex in the left figure in
Figure~\ref{fig.transformation_2} and then
identify the two white vertices.
The resulting diagram is illustrated
in the right figure in Figure~\ref{fig.transformation_2},
where the black vertex is the image of the white vertices.
It should be noted that
this modification does not change the boundary labels of $M$
and the new vertex of $M$ is converging or diverging.

\begin{figure}[h]
\includegraphics{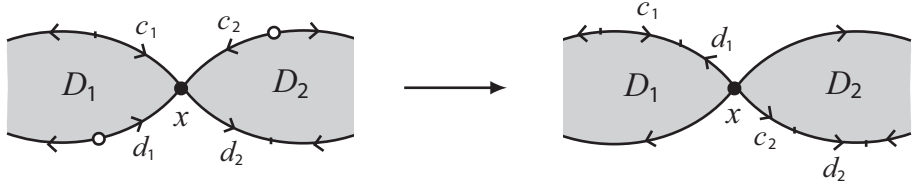}
\caption{
The transformation of Figure~\ref{fig.vertex_type}(c) when
$c_2=d_1$ so that $x$ is diverging}
\label{fig.transformation_2}
\end{figure}

Assume that $x$ is depicted as in Figure~\ref{fig.vertex_type}(d).
Then $c_1+c_2$ is a term of $CS(\phi(\alpha))=CS(s)$
and $(d_1, d_2)$ is a subsequence of $CS(\phi(\delta))=CS(s')$.
Thus each of $c_1+c_2$, $d_1$ and $d_2$ is either $m$ or $m+1$ by Lemma~\ref{lem:terms_II}.
Moreover, $c_1+d_1$ is a term of $CS(u_r^n)$
and hence it is either $m$ or $m+1$.
So, we have the following two possibilities:
\begin{enumerate}[\indent \rm (i)]
\item $c_1=1$, $c_2=m$, $d_1=m$, $d_2=m$;

\item $c_1=1$, $c_2=m$, $d_1=m$, $d_2=m+1$.
\end{enumerate}
In either case, since $c_2=d_1$,
we can transform $M$ so that $x$ is converging
as in Figure~\ref{fig.transformation_1}.

\begin{figure}[h]
\includegraphics{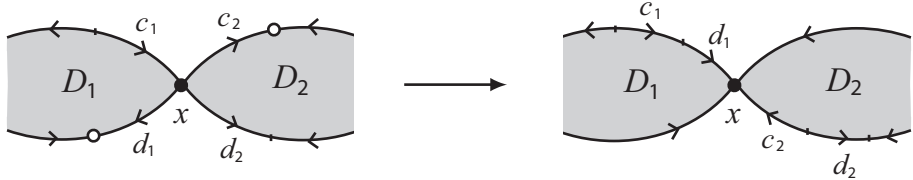}
\caption{
The transformation of Figure~\ref{fig.vertex_type}(d) when $c_2=d_1$
so that $x$ is converging}
\label{fig.transformation_1}
\end{figure}

Assume that $x$ is depicted as in Figure~\ref{fig.vertex_type}(e).
Then $c_1+c_2$ is a term of $CS(\phi(\alpha))=CS(s)$
and $d_1+d_2$ is a term of $CS(\phi(\delta))=CS(s')$.
Thus each of $c_1+c_2$ and $d_1+d_2$ is either $m$ or $m+1$ by Lemma~\ref{lem:terms_II}.
Moreover, for each $i=1,2$, $c_i+d_i$ is a term of $CS(u_r^n)$ and hence
it is either $m$ or $m+1$.
So, we have the following six possibilities:
\begin{enumerate}[\indent \rm (i)]
\item $c_1+c_2=m$, $d_1+d_2=m$, $c_1+d_1=m$, $c_2+d_2=m$;

\item $c_1+c_2=m$, $d_1+d_2=m+1$, $c_1+d_1=m$, $c_2+d_2=m+1$;

\item $c_1+c_2=m$, $d_1+d_2=m+1$, $c_1+d_1=m+1$, $c_2+d_2=m$;

\item $c_1+c_2=m+1$, $d_1+d_2=m$, $c_1+d_1=m$, $c_2+d_2=m+1$;

\item $c_1+c_2=m+1$, $d_1+d_2=m$, $c_1+d_1=m+1$, $c_2+d_2=m$;

\item $c_1+c_2=m+1$, $d_1+d_2=m+1$, $c_1+d_1=m+1$, $c_2+d_2=m+1$.
\end{enumerate}
If (i), (ii), (v) or (vi) happens, then
$c_1+c_2=c_1+d_1$ and so $c_2=d_1$.
Thus, as illustrated in Figure~\ref{fig.transformation_1},
we may transform $M$ so that $x$ is converging.
If (iii) or (iv) happens, then
$c_1+c_2=c_2+d_2$ and so $c_1=d_2$.
So we can transform $M$ so that $x$ is diverging
as in Figure~\ref{fig.transformation_2}.
\end{proof}

\begin{lemma}
\label{lem:twist_link}
$r \neq [m,2]$.
\end{lemma}

\begin{proof}
Suppose on the contrary that $r=[m,2]$.
By \cite[Lemma~3.12(3)]{lee_sakuma_9}, $S_1=(m+1)$ and $S_2=(m)$.

\medskip
\noindent {\bf Claim.} $s, s' \notin I_1(r;n) \setminus I_1(r)$.

\begin{proof}[Proof of Claim]
Suppose on the contrary that $s$ or $s'$ is contained in $I_1(r;n) \setminus I_1(r)$.
Without loss of generality, assume that $s \in I_1(r;n) \setminus I_1(r)$,
i.e., $[m+1] < s \le [m,1,2]$.
Suppose that $s=[m,1,2]$.
Then
$CS(s)=\lp 2\langle m+1\rangle , m, 2\langle m+1\rangle, m \rp$
(see \cite[Lemma 3.12(3)]{lee_sakuma_9}).
Since $CS(\phi(\partial D))=\lp 2n \langle m+1, m \rangle \rp$ for every face $D$ in $M$,
we see by Lemma~\ref{lem:converging_diverging} that
there are two successive $2$-cells, say $D_1$ and $D_2$, in $M$
such that $S(\phi(\partial D_1^+))=(\dots, m+1)$,
$S(\phi(\partial D_2^+))=(m+1, \dots)$ and such that $S(\phi(\partial D_1^+\partial D_2^+))=(\dots, m+1, m+1, \dots)$.
Since neither $S(\phi(\partial D_1^-))$ nor $S(\phi(\partial D_2^-))$ can contain
$((2n-2) \langle m+1, m \rangle, m+1)$ as a subsequence
by Lemma~\ref{lem:connection}(1),
this together with Lemma~\ref{lem:converging_diverging}
implies that
$S(\phi(\partial D_1^+))$ contains $(m+1,m,m+1)$ as a tail
and $S(\phi(\partial D_2^+))$ contains $(m+1,m,m+1)$ as a head.
Since $CS(s)=\lp 2\langle m+1\rangle , m, 2\langle m+1\rangle, m \rp$,
this implies that $M$ consists of only two $2$-cells $D_1$ and $D_2$
with $S(\phi(\partial D_1^+))=(m+1,m,m+1)$ and
$S(\phi(\partial D_2^+))=(m+1,m,m+1)$.
Hence, $CS(\phi(\delta^{-1}))=CS(\phi(\partial D_1^-\partial D_2^-))=
\lp(2n-2) \langle m, m+1 \rangle, m, (2n-2) \langle m, m+1 \rangle, m \rp$.
Since $CS(s')=CS(\phi(\delta^{-1}))$,
this implies that $s'=[m,2,2n-2]$.
But then $s'$ is not contained in $I_1(r;n) \cup I_2(r;n)$,
a contradiction.

So assume that $[m+1] < s < [m,1,2]$.
Write $s$ as a continued fraction expansion
$s=[l_1, l_2, \dots, l_h]$,
where $h \ge 1$, $(l_1, \dots, l_h) \in (\ZZ_+)^h$ and $l_h \ge 2$.
Then $l_1=m$, $l_2=1$ and either $l_3 \ge 3$ or both $l_3 \ge 2$ and $h \ge 4$.
In either case,
we can see by using \cite[Lemma~3.12]{lee_sakuma_9} that
$CS(\phi(\alpha))=CS(s)$ contains $(m+1,m+1,m+1)$ as a subsequence.
Then, since $m+1$'s are isolated in
$CS(\phi(\partial D))=\lp 2n \langle m+1, m \rangle \rp$ for every face $D$ in $M$,
we see by Lemma~\ref{lem:converging_diverging} that
there are three successive faces $D_1, D_2, D_3$ in $M$
such that $S(\phi(\partial D_1^+))=(\dots, m+1)$, $S(\phi(\partial D_2^+))=(m+1)$,
$S(\phi(\partial D_3^+))=(m+1, \dots)$ and such that
$S(\phi(\partial D_1^+\partial D_2^+\partial D_3^+))=(\dots, m+1, m+1, m+1, \dots)$.
But then, since $CS(\phi(\partial D_2))=\lp 2n \langle m+1, m \rangle \rp$,
we have
$S(\phi(\partial D_2^-))=(m, (2n-1) \langle m+1, m \rangle)$.
Hence $CS(\phi(\delta))=CS(s')$ contains $((2n-2) \langle m+1, m \rangle, m+1)$ as a subsequence,
which is a contradiction to Lemma~\ref{lem:connection}(1).
\end{proof}

By the above claim, we have
both $s$ and $s'$ belong to $I_2(r;n) \setminus I_2(r)$.
Then by Lemma~\ref{lem:outside_orbit}(1),
both $CS(s)$ and $CS(s')$ contain
$(m, d \langle m, m+1 \rangle, m, m)$
where $1 \le d \le 2n-3$, because $S_{1e}=S_{1b}=\emptyset$.
Again by Lemma~\ref{lem:converging_diverging}, we see that
there are two successive $2$-cells, say $D_1$ and $D_2$, in $M$
such that $S(\phi(\partial D_1^+))=(\dots, m)$,
$S(\phi(\partial D_2^+))=(m, \dots)$
and such that $S(\phi(\partial D_1^+\partial D_2^+))=(\dots, m, m, \dots)$.
Then since $CS(\phi(\partial D_1))=CS(\phi(\partial D_2))=\lp 2n \langle m+1, m \rangle \rp$,
we have
$S(\phi(\partial D_1^-))=(\dots, m+1)$, $S(\phi(\partial D_2^-))=(m+1, \dots)$
and $S(\phi(\partial D_1^-\partial D_2^-))=(\dots, m+1, m+1, \dots)$.
This implies that $CS(\phi(\delta))=CS(s')$ contains $(m+1, m+1)$.
But, $CS(s')$ also contains $(m,m)$ as observed at the beginning of this paragraph,
a contradiction to \cite[Lemma~3.5]{lee_sakuma_9}.
\end{proof}

By $\tilde{r}$, $\tilde{s}$ and $\tilde{s}'$,
we denote the rational numbers defined as in \cite[Lemma~3.8]{lee_sakuma_9}
for the rational numbers $r, s$ and $s'$
so that $CS(\tilde{r})=CT(r)$, $CS(\tilde{s})=CT(s)$ and $CS(\tilde{s}')=CT(s')$.

\begin{lemma}
\label{lem:claim3_II}
$\tilde{s}, \tilde{s}' \in I_1(\tilde{r};n)\cup I_2(\tilde{r};n)$.
\end{lemma}

\begin{proof}
Note that $s$ and $s'$ belong to
$(I_1(r;n) \setminus I_1(r))\cup (I_2(r;n) \setminus I_2(r))$ by Lemma~\ref{lem:claim2_II},
and therefore $CS(s)$ and $CS(s')$ contain both $S_1$ and $S_2$.
Write $s,s'$ as continued fraction expansions
$s=[p_1, p_2, \dots, p_h]$ and $s'=[q_1, q_2, \dots, q_l]$,
where $p_i, q_j \in \ZZ_+$ and $p_h, q_l \ge 2$.
Since both $CS(s)$ and $CS(s')$ consist of $m$ and $m+1$ by Lemma~\ref{lem:terms_II},
we have $p_1=q_1=m$ by \cite[Lemma~3.5]{lee_sakuma_9}.
If $m_2=1$, then by \cite[Corollary~3.14(1)]{lee_sakuma_9},
$(m+1,m+1)$ appears in $S_1$, so in $CS(s)$ and $CS(s')$.
This implies by \cite[Lemma~3.5]{lee_sakuma_9} that $p_2=q_2=1$.
Also if $m_2 \ge 2$, then by \cite[Corollary~3.14(2)]{lee_sakuma_9}
together with Lemma~\ref{lem:twist_link},
$(m,m)$ appears in $S_2$, so in $CS(s)$ and $CS(s')$.
This implies by \cite[Lemma~3.5]{lee_sakuma_9} that $p_2, q_2 \ge 2$.

Therefore, by \cite[Lemma~3.8]{lee_sakuma_9},
if $\tilde{r}=[m_3, \dots, m_k]$
then $\tilde{s}=[p_3, \dots, p_h]$ and $\tilde{s}'=[q_3, \dots, q_l]$,
whereas
if $\tilde{r}=[m_2-1, m_3, \dots, m_k]$
then $\tilde{s}=[p_2-1, p_3, \dots, p_h]$ and $\tilde{s}'=[q_2-1, q_3, \dots, q_l]$.
This together with the fact $p_1=q_1=m$
and $p_2=q_2$ yields the assertion.
\end{proof}

\begin{lemma}
\label{lem:key_ending}
The unoriented loops
$\alpha_{\tilde{s}}$ and $\alpha_{\tilde{s}'}$
represent the same conjugacy class in $\Hecke(\tilde{r};n)$.
\end{lemma}

\begin{proof}
Let $\tilde{R}$ be the symmetrized subset of $F(a, b)$ generated by the single relator
$u_{\tilde{r}}^n$ of the upper presentation $\Hecke(\tilde{r};n)=\langle a, b \, | \, u_{\tilde{r}}^n \rangle$.
In the following, we construct an annular
$\tilde{R}$-diagram $(\tilde{M},\psi)$
from the given $R$-diagram $(M,\phi)$
such that $u_{\tilde{s}}$ is an outer boundary label and
$u_{\tilde{s}'}^{\pm 1}$ is an inner boundary label of $\tilde{M}$.
To this end, recall that $M$ is shaped as in \cite[Figure~3(a)]{lee_sakuma_9}.
By Lemma~\ref{lem:converging_diverging}, we assume that
every vertex of $M$ with degree $4$ is either converging or diverging.
Then, under \cite[Notation~4.18]{lee_sakuma_9},
each $S(\phi(\partial D_i^{\pm}))$ consists of $m$ and $m+1$,
and moreover, it does not contain $(m,m)$ or $(m+1,m+1)$
according to whether $m_2=1$ or $m_2\ge 2$.
Thus the $T$-sequence of $\phi(\partial D_i^{\pm})$ is defined as in
\cite[Definition~3.6]{lee_sakuma_9}
by counting the numbers of consecutive $m+1$'s or $m$'s
according to whether $m_2=1$ or $m_2\ge 2$.
For the precise definition of $T$-sequences,
see \cite[Definitions 8.4 and 8.7]{lee_sakuma_4}.

Now let $\tilde M$ be the map obtained from $M$
by forgetting all degree $2$ vertices, and let $\tilde D_i$ and $\partial \tilde D_i^{\pm}$,
respectively,
be the copies of $D_i$ and $\partial D_i^{\pm}$ ($1 \le i\le t$).
For each $i$ with ($1 \le i\le t$), we assign an alternating word,
$\psi(\partial \tilde D_i^{\pm})$, in $\{a,b\}$, to $\partial \tilde D_i^{\pm}$
as follows.

{\bf Step 1.}
For $i=1,\dots, t$, assign $\psi(\partial\tilde{D}_i^{+})$
so that
$\psi(\partial\tilde D_1^+ \cdots \partial \tilde{D}_i^+)
:=\psi(\partial\tilde D_1^+) \cdots \psi(\partial\tilde{D}_i^+)$
is alternating and
\[
S(\psi(\partial\tilde D_1^+ \cdots \partial\tilde D_i^+))
=T(\phi(\partial D_1^+ \cdots \partial D_i^+))
\]
Once this assignment is done, we see the following.
\begin{enumerate}[\indent \rm (i)]
\item
The word $\psi(\partial\tilde D_1^+ \cdots \partial\tilde D_t^+)$ is cyclically alternating,
because the sum of the terms of $CT(s)=CS(\tilde{s})$ is even.
\item
$CS(\psi(\partial\tilde D_1^+ \cdots \partial\tilde D_t^+))=
CT(\phi(\partial D_1^+ \cdots \partial D_t^+))=CT(\phi(\alpha))=CT(s)=CS(\tilde{s})$,
because $CS(\tilde{s})$ has even number of terms.
In particular, $(\psi(\tilde\alpha)) \equiv (u_{\tilde s}^{\pm 1})$ by \cite[Lemma~5.2]{lee_sakuma},
where $\tilde\alpha$ is an outer boundary cycle of $\tilde M$.
\end{enumerate}

{\bf Step 2.}
For $i=1,\dots, t$, assign $\psi(\partial\tilde D_i^-)$ so that
$\psi(\partial\tilde D_i):=\psi(\partial\tilde D_i^+)\psi(\partial\tilde D_i^-)^{-1}$ is an alternating word
and
$S(\psi(\partial\tilde D_i^-))=T(\phi(\partial D_i^-))$.
Once this assignment is done, we see the following.
\begin{enumerate}[\indent \rm (i)]
\item
For $i=1,\dots, t$,
$\psi(\partial\tilde D_1^- \cdots \partial\tilde D_i^-)$ is a reduced alternating word
such that
$S(\psi(\partial\tilde D_1^- \cdots \partial\tilde D_i^-))
=T(\phi(\partial D_1^- \cdots \partial D_i^-))$.
\item
The word $\psi(\partial\tilde D_1^- \cdots \partial\tilde D_t^-)$ is cyclically alternating,
because the sum of the terms of $CT(s')=CS(\tilde{s}')$ is even.
\item
$CS(\psi(\partial\tilde D_1^- \cdots \partial\tilde D_t^-))=
CT(\phi(\partial D_1^- \cdots \partial D_t^-))=CT(\phi(\delta^{-1}))=CT(s')=CS(\tilde{s}')$,
because $CS(\tilde{s})$ has even number of terms.
In particular, $(\psi(\tilde\delta)) \equiv (u_{\tilde{s}'}^{\pm 1})$ by \cite[Lemma~5.2]{lee_sakuma},
where $\tilde\delta$ is an inner boundary cycle of $\tilde M$.
\end{enumerate}
Here the assertion (i) is proved as follows,
as in the proof of assertion (i) in
\cite[Step~3 in Section~8.1]{lee_sakuma_4}.
We assume $m_2 \ge 2$ and verify the assertion when $i=2$.
(The other cases can be treated similarly.)
Since both $CS(\phi(\alpha))=CS(s)$ and
$CS(\phi(\partial D_1))=CS(\phi(\partial D_2))=\lp 2n\langle S_1,S_2\rangle \rp$
consist of $m$ and $m+1$ and do not contain $(m+1,m+1)$,
we have four possibilities around the vertex between $D_1$ and $D_2$
as described in the left figures in Figure~\ref{fig.taking_tsequence_1},
up to reflection in the vertical line passing through the vertex.
In each of the right figure, we may assume without loss of generality
that the upper left segment is oriented
so that it is converging into the vertex.
Then the orientations of the three remaining segments
in each of the right figures are specified
by the requirements in Step~1 and the new requirement
$S(\psi(\partial\tilde D_i^-))=T(\phi(\partial D_i^-))$ for $i=1,2$.
In each case, we can check that the condition
$S(\psi(\partial\tilde D_1^-\partial\tilde D_2^-))
=T(\phi(\partial D_1^-\partial D_2^-))$ holds.

\begin{figure}[h]
\includegraphics{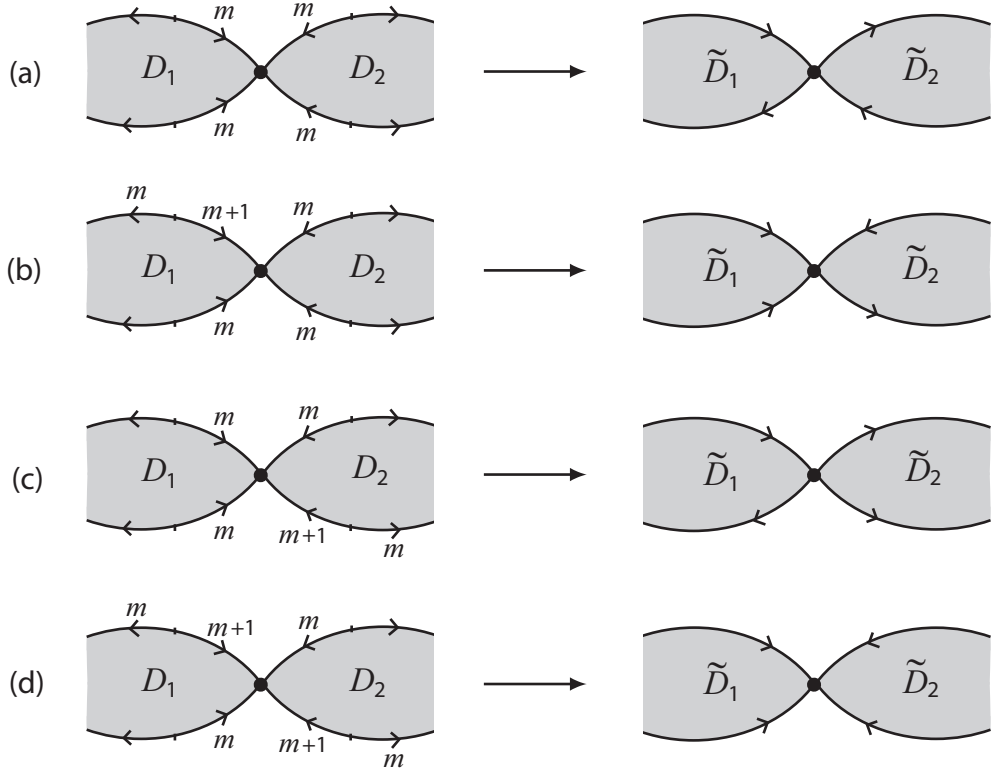}
\caption{
The construction of $\tilde{M}$ from $M$}
\label{fig.taking_tsequence_1}
\end{figure}

By the annular $\tilde{R}$-diagram $(\tilde M,\psi)$ constructed in the above,
we see that the unoriented loops
$\alpha_{\tilde{s}}$ and $\alpha_{\tilde{s}'}$
represent the same conjugacy class in $\Hecke(\tilde{r};n)$
(cf. \cite[Lemma~4.11]{lee_sakuma_9}).
\end{proof}

We repeatedly apply Lemmas~\ref{lem:claim3_II} and \ref{lem:key_ending}
to obtain a contradiction as follows.
Recall that
$r=[m_1, \dots, m_k]$ with $m_1=m \ge 2$, $m_k\ge 2$ and $k \ge 2$.
Thus $\tilde r$ is equal to
$[m_2-1,m_3,\dots, m_k]$ or $[m_3,\cdots, m_k]$
according to whether $k\ge 3$ or both $m_2=1$ and $k\ge 3$.
In particular, $0<\tilde r \le 1$.
Since $r\ne [m,2]$ by Lemma~\ref{lem:twist_link},
we must have $0<\tilde r<1$.
If $1/2<\tilde r<1$, then $0<1-\tilde r <1/2$ and,
by \cite[Lemma~2.5]{lee_sakuma_9},
there is an orbifold homeomorphism $f$ from
$\orbs(\tilde r;n)$ to $\orbs(1-\tilde r;n)$
which maps the 2-bridge sphere of $\orbs(\tilde r;n)$ to
that of $\orbs(1-\tilde r;n)$,
such that the restriction of $f$ to the $2$-bridge sphere
maps the simple loop $\alpha_s$ to $\alpha_{1-s}$
for any $s\in\QQQ$.
Moreover the transformation $s\mapsto 1-s$ maps
$I(\tilde r;n)$ to $I(1-\tilde r;n)$.
So we assume $0<\tilde r\le 1/2$.
If $\tilde r = 1/p$ for some $p\ge 2$,
then we have a contradiction
by virtue of \cite[Main Theorem~2.5(1)]{lee_sakuma_9}.
So, we may assume $\tilde r \ne 1/p$,
and therefore, we can apply Lemmas~\ref{lem:claim3_II} and \ref{lem:key_ending}
to $\tilde r$.
By repeating this argument,
we finally arrive at the situation that
for either $r'=1/p$ or $r'=[p,2]$ for some integer $p \ge 2$,
there are two rational numbers $t, t' \in I_1(r';n) \cup I_2(r';n)$
for which the simple loops $\alpha_{t}$ and $\alpha_{t'}$ represent the same conjugacy class in $\Hecke(r';n)$.
The former is a contradiction to \cite[Main Theorem~2.5(1)]{lee_sakuma_9},
and the latter is a contradiction to Lemma~\ref{lem:twist_link}.
This completes the proof of Main Theorem~\ref{thm:conjugacy}(1).
\qed

\section{Proof of Main Theorem~\ref{thm:conjugacy}(2) and (3)}
\label{sec:proof of main theorem(3)}

Main Theorem~\ref{thm:conjugacy}(2) can be proved
by simply replacing $1/p$ with a non-integral rational number $r$
in \cite[Section~7]{lee_sakuma_9}.
The only difference is to use Corollary~\ref{cor:consecutive_vertices},
instead of \cite[Corollary~5.2]{lee_sakuma_9},
at the end of the proof.

It remains to prove Main Theorem~\ref{thm:conjugacy}(3).
Let $S(r)= (S_1, S_2, S_1, S_2)$ be as in \cite[Lemma~3.9]{lee_sakuma_9}.
We recall the following lemma.

\begin{lemma}
\label{lem:half}
{\rm (1)}
Suppose that $v$ is a cyclically alternating word which represents the trivial element
in $G(K(r))=\langle a, b \, | \, u_r \rangle$.
Then the cyclic word $(v)$ contains a subword $w$ of the cyclic word
$(u_r^{\pm 1})$ such that $S(w)$ is $(S_1, S_2, \ell)$ or $(\ell, S_2, S_1)$
for some $\ell \in \ZZ_+$.

{\rm (2)}
Suppose that $v$ is a cyclically alternating word which represents the trivial element
in $\Hecke(r;n)=\langle a, b \, | \, u_r^n \rangle$.
Then the cyclic word $(v)$ contains a subword $w$ of the cyclic word
$(u_r^{\pm n})$ such that
$S(w)$ is $((2n-1)\langle S_1, S_2 \rangle, \ell)$ or
$(\ell, (2n-1)\langle S_2, S_1 \rangle)$, where $\ell \in \ZZ_+$.
\end{lemma}

\begin{proof}
(1) This is nothing other than \cite[Theorem~6.3]{lee_sakuma}.

(2) By the first assertion of
\cite[Corollary~4.12]{lee_sakuma_7},
we see that the cyclic word $(v)$ contains a subword $w$ of the cyclic word
$(u_{r}^{\pm n})$
which is a product of $4n-1$ pieces but is not a product of
less than $4n-1$ pieces.
Hence, we obtain the desired result by \cite[Lemma~4.3(2)]{lee_sakuma_9}.
\end{proof}

Suppose on the contrary that there is a rational number $s$ in $I_1(r;n) \cup I_2(r;n)$
for which $u_s^t=1$ in $\Hecke(r;n)$ for some integer $t \ge 1$.
Then clearly $u_s^t=1$ also in $G(K(r))$.
Since $G(K(r))$ is torsion-free, $u_s=1$ in $G(K(r))$.
By \cite[Main Theorem~2.3]{lee_sakuma},
this implies that $s$ lies in the $\RGPP{r}$-orbit of $r$ or $\infty$.
Hence $|u_s|> |u_r|$.

On the other hand, since $u_s^t=1$ in $\Hecke(r;n)$,
Lemma~\ref{lem:half}(2) implies that
we may write $\bar{u}_s^t \equiv wz$, where $\bar{u}_s$ is a cyclic permutation of $u_s$
and $w$ is a subword of $(u_s^t)$ as described in Lemma~\ref{lem:half}(2).
If $w$ is contained in $\bar{u}_s$,
then by \cite[Lemma~4.3(3)]{lee_sakuma_9},
$CS(s)=CS(\bar{u}_s)$ contains $((2n-1)\langle S_1, S_2 \rangle)$ or
$((2n-1)\langle S_2, S_1 \rangle)$ as a subsequence.
But, since $s \in I_1(r;n) \cup I_2(r;n)$, this is impossible by
Lemma~\ref{lem:connection}.
So $w$ cannot be contained in $\bar{u}_s$, and hence
$\bar{u}_s$ is a proper initial subword of $w$.
Thus
$\bar{u}_s$ is a subword of the cyclic word $(u_r^{\pm n})$.
Here, since $|\bar{u}_s|=|u_s| > |u_r|$,
we may put $\bar{u}_s \equiv v^d v_1$,
where $d \in \ZZ_+$, $v$ is a cyclic permutation of $u_r$ or $u_r^{-1}$
and $|v_1| <|u_r|$.
Note that $|v_1| \ge 1$, for otherwise we would have $\bar{u}_s \equiv v^d$
so that $CS(s)=\lp 2d \langle S_1, S_2 \rangle \rp$,
which yields that $d \ge 2$ since $s \neq r$
and that $s=dq/dp$ if $r=q/p$ by \cite[Remark~3.11]{lee_sakuma_9},
a contradiction.

Then $v_1=\bar{u}_s=1$ in $G(K(r))$.
Moreover, $v_1$ is a proper initial subword of $v$
and so a proper initial subword of $\bar{u}_s$.
This implies that $v_1$ is cyclically alternating,
because $\bar{u}_s$ is cyclically alternating and $|v_1|=|u_s|-d|u_r|$ is even.
Also since $v_1=1$ in $G(K(r))$, Lemma~\ref{lem:half}(1) implies that
the cyclic word $(v_1)$ contains a proper subword $w$
such that $S(w)$ is $(S_1, S_2)$ or $(S_2, S_1)$.
So $|v_1| > \frac{1}{2} |u_r|=\frac{1}{2} |v|$.
Now, let $v_2$ be the terminal subword of $v$ such that $v\equiv v_1v_2$.
Then $v_2$ is also cyclically alternating and $v_2=1$ in $G(K(r))$.
Thus the above argument implies that $|v_2| > \frac{1}{2} |v|$.
This contradicts the inequality $|v_2| =|v|-|v_1|<\frac{1}{2} |v|$.
\qed

\bibstyle{plain}

\end{document}